\documentclass[12pt]{article}

\usepackage{amsfonts,amsthm,amsmath,latexsym,amssymb}

\usepackage[dvips]{graphics}

\usepackage{epsfig}
\usepackage{graphicx}
\usepackage{color}
\usepackage{hyperref}
\newcommand{\norm}[1]{\lVert#1\rVert}
\newcommand{\abs}[1]{\lvert#1\rvert}

\usepackage{epsf}
\usepackage{psfrag}
\newdimen\epsfxsize
\newdimen\epsfysize

\input epsf.tex
\newdimen\epsfxsize
\newdimen\epsfysize

\newcommand{\be}{\begin{equation}}
\newcommand{\ee}{\end{equation}}
\newcommand{\bes}{\begin{equation*}}
\newcommand{\ees}{\end{equation*}}
\newcommand{\s}{\sigma}

\renewcommand{\l}{\lambda}

\renewcommand{\H}{\mathbb H}

\newcommand{\R}{\mathbb R}

\newcommand{\e}{\epsilon}

\newtheorem{thm}{Theorem}[section]
\newtheorem{prop}[thm]{Proposition}

\newtheorem{cor}[thm]{Corollary}

\newtheorem{lemma}[thm]{Lemma}

\makeatletter

\@addtoreset{equation}{section} \makeatother

\def\1{{\bf 1}}

\begin{document}

\baselineskip=1.2\baselineskip
\begin{doublespace}

\title{\bf Spacefilling Curves and Phases of the Loewner Equation}
\bigskip
\author{{\bf Joan Lind}~{\bf and}  {\bf Steffen Rohde}
\footnote{Research supported by NSF Grant DMS-0800968.}}

\maketitle

\abstract{
Similar to the well-known phases of SLE, the Loewner differential equation with Lip(1/2) driving terms is known to have a phase transition at norm 4, when traces change from simple to non-simple curves. We establish the deterministic analog of the second phase transition of SLE, where traces change to space-filling curves: There is a constant $C>4$ such that a Loewner driving term whose trace is space filling has Lip(1/2) norm at least $C.$  We also provide a geometric criterion for traces to be driven by  Lip(1/2)  functions, and show that for instance the Hilbert space filling curve and the Sierpinski gasket fall into this class.
}

\tableofcontents

\bigskip

\section{Introduction and Results}\label{intro}
The Schramm-Loewner Evolution $SLE_{\kappa}$ is the random process of planar curves generated by the
Loewner equation driven by $\lambda(t) = \sqrt{\kappa} B_t$, where $B_t$ is a standard one-dimensional Brownian motion.
(See Section \ref{background} for definitions and references.) 
It is well-known that SLE exhibits two phase transitions,
namely at $\kappa=4$ and at $\kappa=8$. 
For $\kappa\leq4$, the traces are simple curves, whereas
for $\kappa>4$ the traces have self-touchings. 
For $\kappa<8,$ the traces have empty interior, whereas for $\kappa \geq 8,$ the traces are spacefilling.
(Note that all statements about SLE hold almost surely.)

\noindent

There is a close analogy between the behaviour of $SLE_{\kappa}$ traces and the hulls of the deterministic
Loewner equation driven by functions $\lambda\in$ Lip$(1/2)$, where the Lip$(1/2)$ norm $\norm{\l}_{1/2}$ plays the role
of $\kappa.$ It is known that the first phase transition of SLE has a deterministic counterpart: If 
$\norm{\l}_{1/2}<4,$ then the trace is a simple curve, whereas there are functions of norm $4$ that generate non-simple curves. In this paper, we prove the existence of a second phase transition for the deterministic Loewner equation:

\begin{thm} \label{2ndphase}
Suppose $\l$ is a Lip$(1/2)$ driving function that generates a curve with non-empty interior.  Then $\norm{\l}_{1/2} \geq 4.0001.$ 
\end{thm}

The constant 4.0001 that we obtain is certainly far from optimal.
The proof of Theorem \ref{2ndphase} relies on a careful analysis of the behaviour of $\lambda$ at times $t$ when
the trace hits the real line. At each of these times, the local Lip$(1/2)$ norm is at least 4. Space filling curves
have uncountably many such times, which we prove will result in a norm bounded away from 4. 
If we give up the requirement of the trace being space filling and only require the trace to be dense in the upper half plane, Theorem \ref{2ndphase} is no longer true:

\begin{prop} \label{counterexpl} 
For every sequence $z_1,z_2,z_3,...$ of points in $\H$,
there is a trace $\gamma$ that visits these points (in this order) and has Lip$(1/2)$ norm at most $4.$
\end{prop}

\noindent

It is not obvious that there are space-filling curves generated by Lip$(1/2)$ driving terms. However, we provide a
rather general criterion for hulls to be driven by Lip$(1/2)$ functions and obtain a large class of examples of such
hulls (including those shown in Figure \ref{corfig}, the classical van Koch curve,  the half-Sierpinski gasket, and the Hilbert space-filling curve):

\begin{thm}\label{qdisc}

Let $\{K_t\}$ be a family of hulls generated by the driving term $\lambda(t)$ for $t \in [0,T]$.    Suppose that  there is some $C_0 > 0$ and some $k < \infty$ so that for each $s \in [0,T)$ there exists a $k$-quasi-disc $D_s \subset \mathbb{H}$ with the following three properties:
\begin{align*}
(1&) \text{ }K_s \subset \H\setminus D_s\\
(2&) \text{ }K_T \setminus K_s \subset \overline{D_s} \\
(3&) \operatorname{ diam}(K_t \setminus K_s) \leq C_0 \max \{\operatorname{dist} (z, \partial D_s)\text{ } | \text{ } z \in K_t \setminus K_s \} \text{ for all } t \in (s,T].
\end{align*}
Then $\lambda$ is in Lip$(1/2)$ and $ \norm{\lambda}_{1/2} \leq C(k,C_0)$.
\end{thm}

See Figure \ref{quasidisc} for an illustration of the required quasi-disc $D_s$ 
(and recall
that a quasi-disc is the image of a disc under a quasiconformal mapping of the plane.)
Intuitively, condition (3) means that the hulls grow ``transversally'' rather than ``tangentially.''
An easy consequence of Theorem \ref{qdisc} is

\begin{cor}\label{expls} The van Koch curve, the half-Sierpinski gasket, and the Hilbert space filling curve 
are all generated by Lip$(1/2)$ driving terms. There is a Lip$(1/2)$ driving term whose trace is a 
simple curve $\gamma$ with positive area. In particular, this $\gamma$ is not conformally removable, 
and therefore not uniquely determined by its conformal welding.

\end{cor}

\noindent

Notice that the space-filling Hilbert curve and the half-Sierpinski gasket and  can be obtained as limits of simple
curves so that their Loewner driving terms are well-defined. See Figures \ref{space} and \ref{halfsierp},
and see Figure \ref{drivefig} for approximations of their driving functions.

\begin{figure}
\centering
\hfill
\includegraphics[scale=0.25]{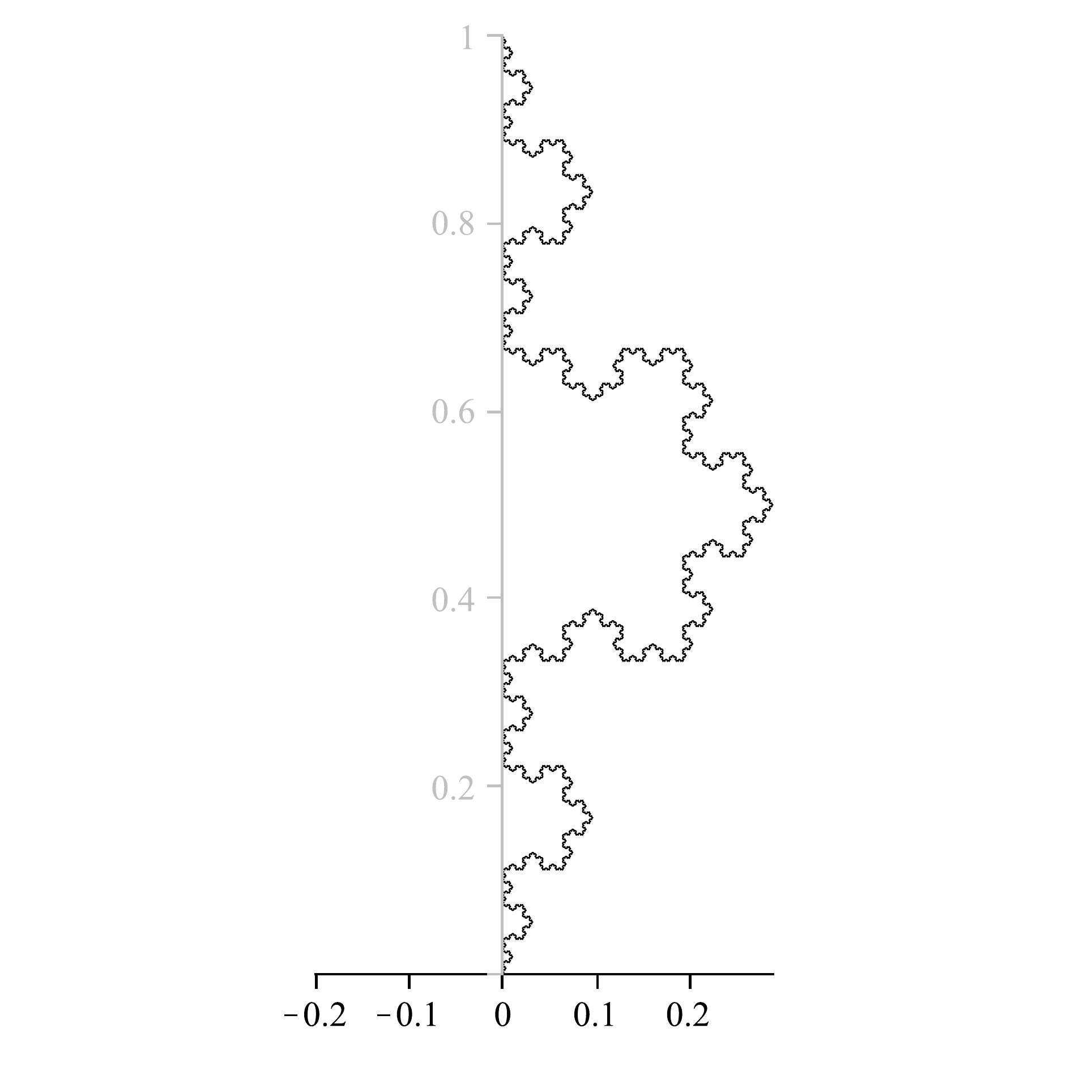}
\includegraphics[scale=0.25]{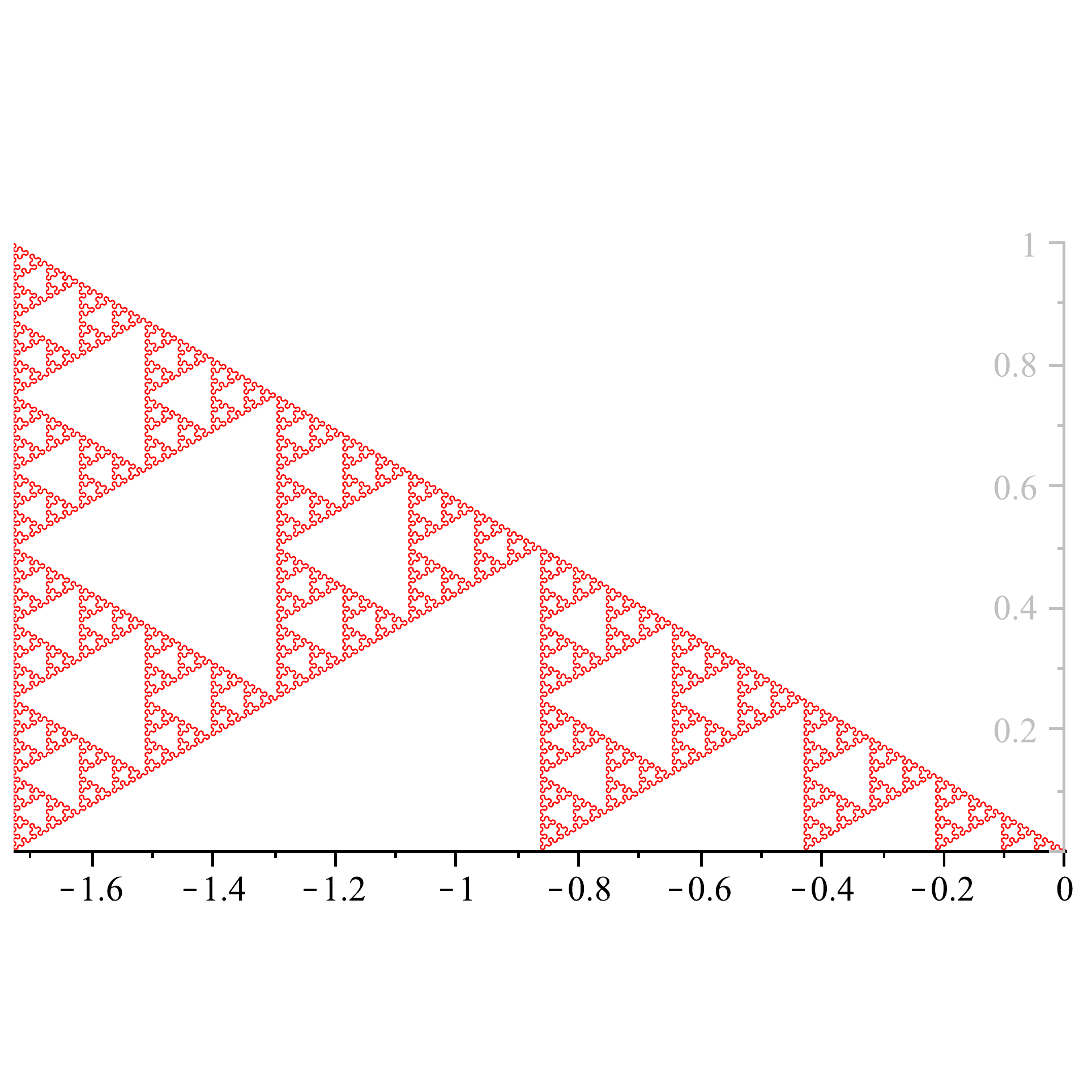}\hfill \hfill \hfill
 \includegraphics[scale=0.25]{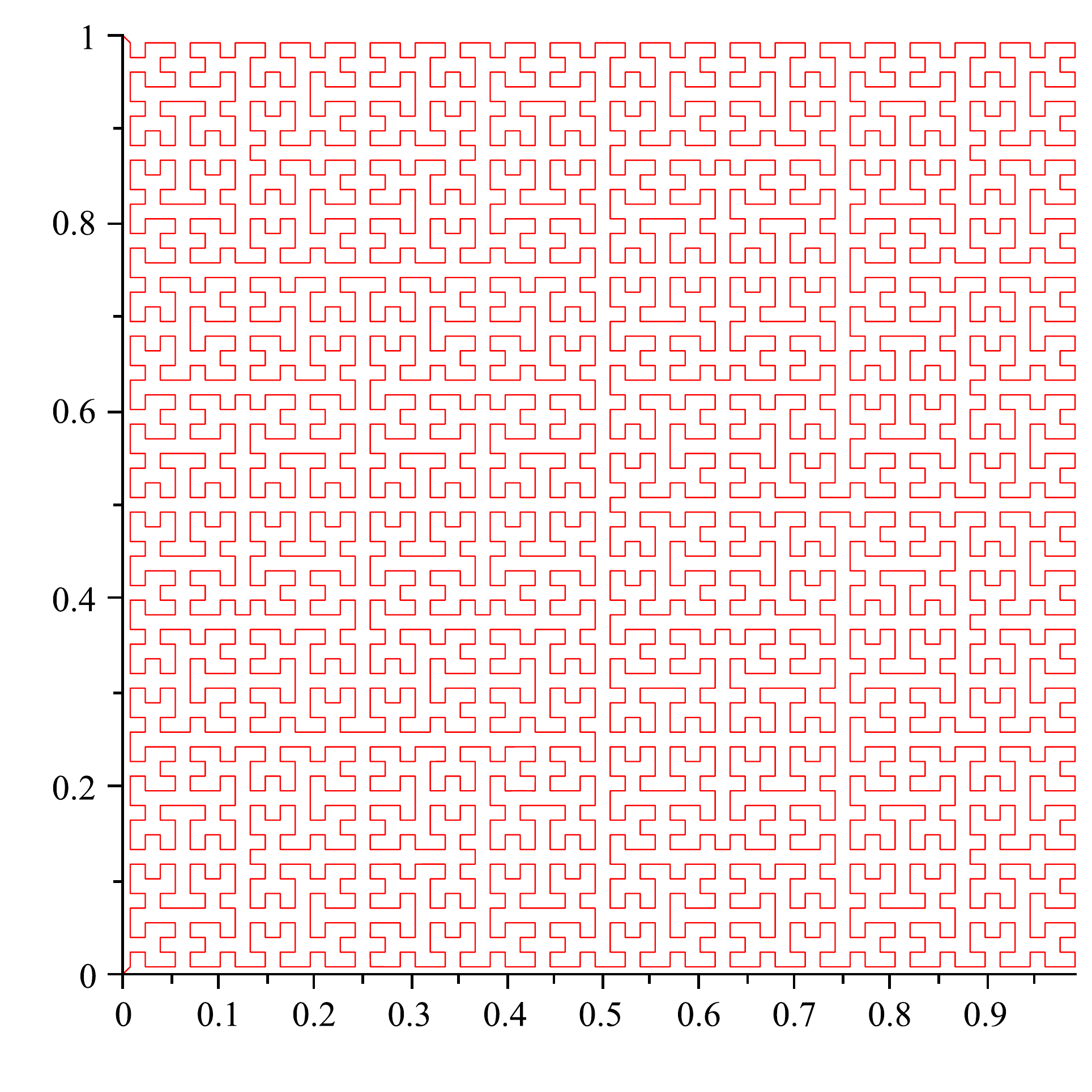}\hfill \hfill \hfill \hfill
\caption{Three curves with Lip(1/2) driving function. 
}\label{corfig}
\end{figure}

\noindent

The organization of the paper is as follows: In Section \ref{background} we review basic definitions,
facts and references and put our results in perspective to SLE. The expert can safely skip it.
Section \ref{phase} contains the proofs of (a slightly stronger version of) Theorem \ref{2ndphase} and Proposition \ref{counterexpl}.
It is independent from
Section \ref{s:domains} where we prove Theorem \ref{qdisc} and its corollary.

{\bf Acknowledgement.}
We thank David White for the pictures of the fractals, and Belmont University undergraduate students Andrew Hill, Matt Lefavor and Ben Stein,  who developed a JAVA  program which can approximate the corresponding driving term.  
We also thank Don Marshall for his comments on a draft of this paper.

\section{Background and Motivation}\label{background}

\subsection{A brief look at the Loewner equation}

In this section, we briefly review the chordal Loewner equation and some of its
standard properties used throughout the paper.

If  $\lambda :[0,T]  \rightarrow \mathbb{R}$ is continuous, then for each 
$z \in \overline{\mathbb{H}} \setminus \{\lambda(0)\}$ the  Loewner equation 
\be\label{le}
\frac{\partial}{\partial t} g_t(z) = \frac{2}{g_t(z)- \lambda(t)} \;\;\; , \;\;\;
g_0(z) = z 
\ee
has a solution on some time interval.
Set $T_z=  \sup\{s \in [0,T]: g_t(z)$ exists on $[0,s)\}$, 
and set $K_t = 
\{z \in \mathbb{H} : T_z \leq t\}$.  
Then  $\mathbb{H} \setminus K_t$ is a simply connected subdomain of $\mathbb{H}$,
 and $g_t$ is the unique conformal map from $\mathbb{H} \setminus K_t$ onto 
$\mathbb{H}$ with the hydrodynamic normalization 
$g_t(z)=z+\frac{2t}{z} + O\left(\frac{1}{z^2}\right)$ near infinity.
The function $\lambda(t)$ is called the driving term, 
and the compact sets $K_t$ are called the hulls generated (or driven) by $\lambda$.  We also consider the domains $\H \setminus K_t$ and the conformal maps $g_t$ to be generated (or driven) by $\l$.  

On the other hand, we may begin with  a sequence of continuously growing hulls $K_t$ with
$K_0=\emptyset$
(see \cite{La} for a precise definition). 
Re-parametrizing $K_t$ as needed, 
we may assume that the conformal maps $g_t \equiv g_{K_t}:\H\setminus K_t \to \H$ have the hydrodynamic normalization at infinity.  Then the maps $g_t$ satisfy the Loewner equation for some continuous
function $\lambda(t)$, and $K_t$ are the hulls generated by $\l(t)$.  Thus the Loewner equation provides a one-to-one correspondence between continuous real-valued functions and certain families of continuously growing hulls.

There is another version of the Loewner equation in the halfplane.  If
$\xi : [0,T] \rightarrow \mathbb{R}$ is continuous 
and $z \in \overline{\mathbb{H}} \setminus \{\xi(0)\}$, 
then the backward Loewner equation 
\be\label{ble}
\frac{\partial}{\partial t} f_t(z) = \frac{-2}{f_t(z)-
\xi(t)} \;\;\; , \;\;\;
f_0(z) = z 
\ee
has a solution on the whole time interval $[0,T]$.
Further, $f_t$ is a
conformal map from $\mathbb{H}$ into $\mathbb{H}$, and near infinity it
has the form  
$f_t(z)=z+\frac{-2t}{z} + O(\frac{1}{z^2}).$
 The two forms of Loewner's differential equation in the halfplane are 
related as follows.  Given  a
continuous function $\lambda$ on $[0,T]$, set $\xi(t)= \lambda(T-t)$.  Let
$g_t$ be the functions generated by $\lambda$ from \eqref{le}, and let
$f_t$ be
the functions generated by $\xi$ from \eqref{ble}.  
Then $f_t =g_{T-t} \circ g_T^{-1}$, and in particular $f_T = g_T^{-1}$.

We mention four simple (but important) properties of the chordal Loewner equation.  Assume that the hulls $K_t$ are generated by the driving term $\l(t)$.  Then
\begin{enumerate}
\item  {\it Scaling}: For $r>0$,
the scaled hulls $\tilde{K}_t := r K_{t/r^2}$ are driven by $r \l(t/ r^2 )$.
\item {\it Translation}: For $x \in \mathbb{R}$,
the driving term of $K_t+x$ is $\l(t)+x$.
\item {\it Reflection}:  The reflected hulls $R_I(K_t)$ are driven by $-\l(t)$, where
$R_I$ denotes reflection in the imaginary axis.  
\item {\it Concatenation}:
For fixed $T$, the mapped hulls $g_T(K_{T+t})$ are driven by $\l(T+t)$.
\end{enumerate}

\subsection{Lip$(1/2)$ driving functions}

A function $\lambda$ belongs to Lip$(1/2)$ if there exists $C >0$ so that
\bes
|\lambda(t) - \lambda(s)| \leq C \, |t-s|^{1/2}
\ees
for all $t,s$ in the domain of $\lambda$.  
The smallest such $C$ is called the Lip$(1/2)$ norm of $\lambda$ and is denoted by $\norm{\lambda}_{1/2}$.
Notice that the Lip$(1/2)$ norm is invariant under the above scaling, i.e.
$\norm{r \l(t/r^2 )}_{1/2}=\norm{\lambda}_{1/2}.$
Thus Lip$(1/2)$ forms a natural class of driving functions for the Loewner equation.
The following was shown in \cite{MR} and \cite{L}:

\begin{thm}\label{LMRthm}
If there exists $k \geq 1$ so that $\H\setminus K_t$ is a $k$-quasislit-halfplane for  all $t \in [0,T]$, 
then $\lambda$ is in Lip$(1/2)$ on $[0,T]$.
Conversely, if $\lambda\in Lip(1/2)$ with $\norm{\lambda}_{1/2}< 4$, 
then $\H\setminus K_t$ is a $k(\norm{\lambda}_{1/2})$-quasislit-halfplane for all $t$.
Further the constant 4 is sharp:~for each $k\geq 4$ there is a Lip$(1/2)$ function with norm $k$ 
that generates a non-slit-halfplane.
\end{thm}

A $k$-quasislit-halfplane is the image of $ \mathbb{H} \setminus [0,i]$
under a $k$-quasiconformal automorphism of $\mathbb{H}$ fixing $\infty$.
For example, it is not hard to show that the complement of the van Koch curve (Figure \ref{corfig})
is a quasislit-halfplane (see \cite{MR}.)

The driving functions $k\sqrt{1-t}$, for $k \geq 4$, are the simplest examples of Lip(1/2) driving functions that do not generate slit-halfplanes for all time.  Rather, for $k\geq 4$ the hull generated by $k\sqrt{1-t}$ is a curve that hits back on the real line and forms a bubble at time 1.  This situation is studied in detail in \cite{KNK} (from a computational viewpoint) and in \cite{LMR} (from a geometric viewpoint.)

In \cite{MR}, there is another example of ``bad'' behavior generated by a Lip$(1/2)$ driving term: a curve which spirals infinitely around a disc.  At the final time, the hull is not even locally connected.  This example can be constructed so that its driving term has Lip$(1/2)$ norm arbitrarily close to 4 (see \cite{LMR}).

\subsection{SLE and self-similar curves}

The three curves in Figure \ref{corfig} are all self-similar. For instance, scaling the whole van Koch curve by $1/3$ gives the first third of the curve, whereas the scaling factor is 1/2 for the half-Sierpinski triangle and 1/4 for the Hilbert curve.
By the scaling property of the Loewner equation, this is reflected in a self-similarity of the driving terms:
For the van Koch curve, $3\lambda(t/9)=\lambda(t),$ as is seen Figure \ref{drivefig}.
The Schramm-Loewner Evolution displays a similar form of self-similarity.

\begin{figure}
\centering
\hfill
\includegraphics[scale=0.25]{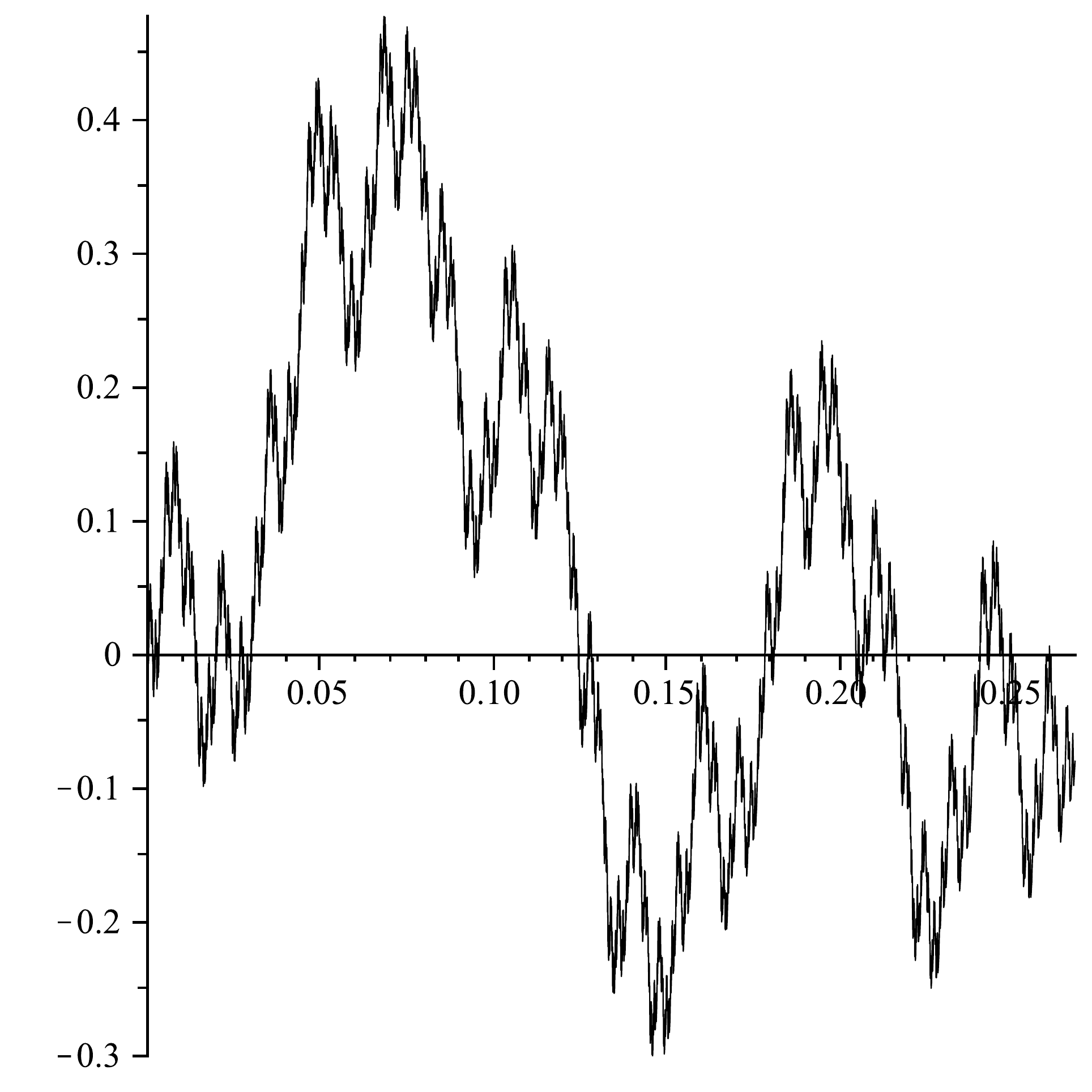}\hfill
\includegraphics[scale=0.25]{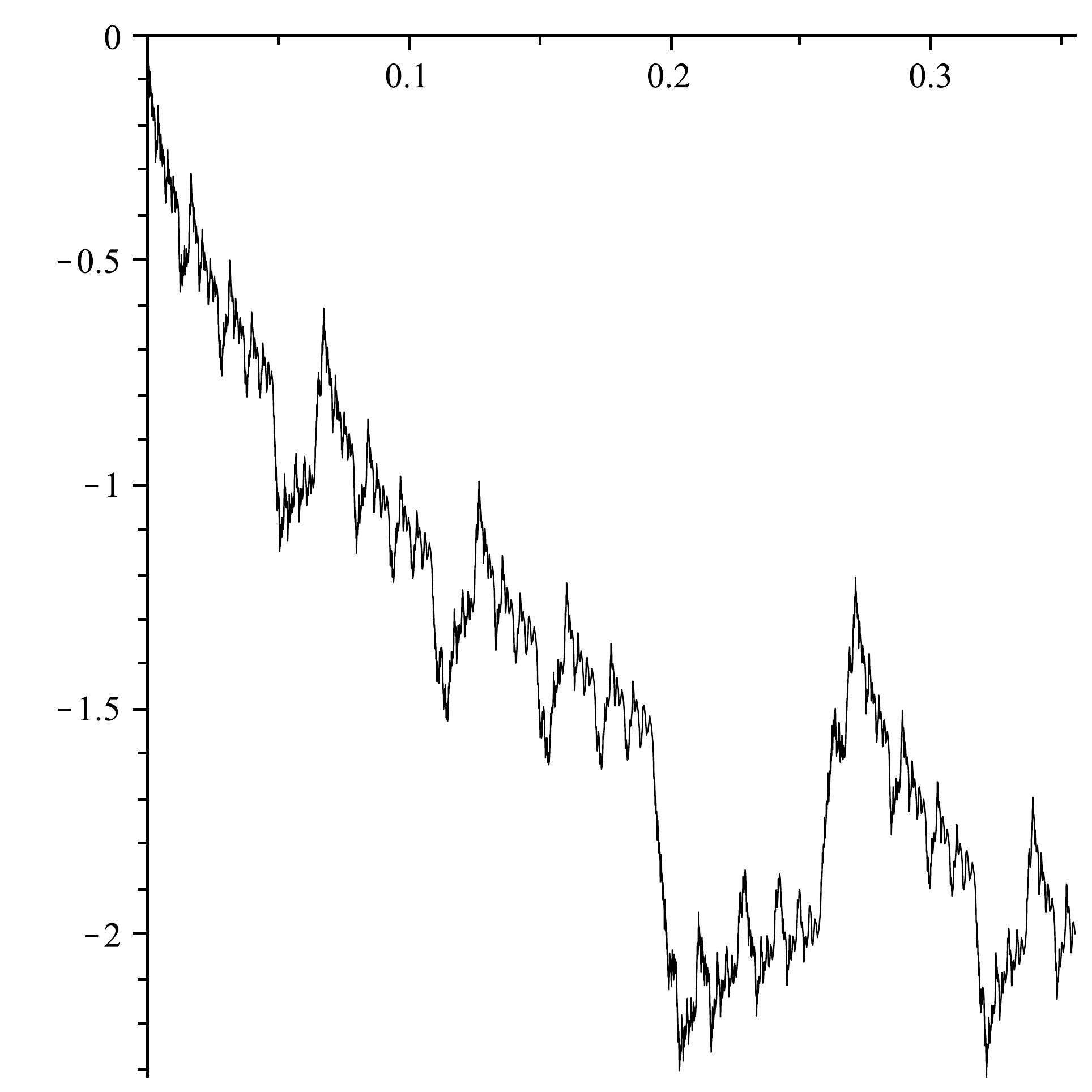}\hfill
\includegraphics[scale=0.25]{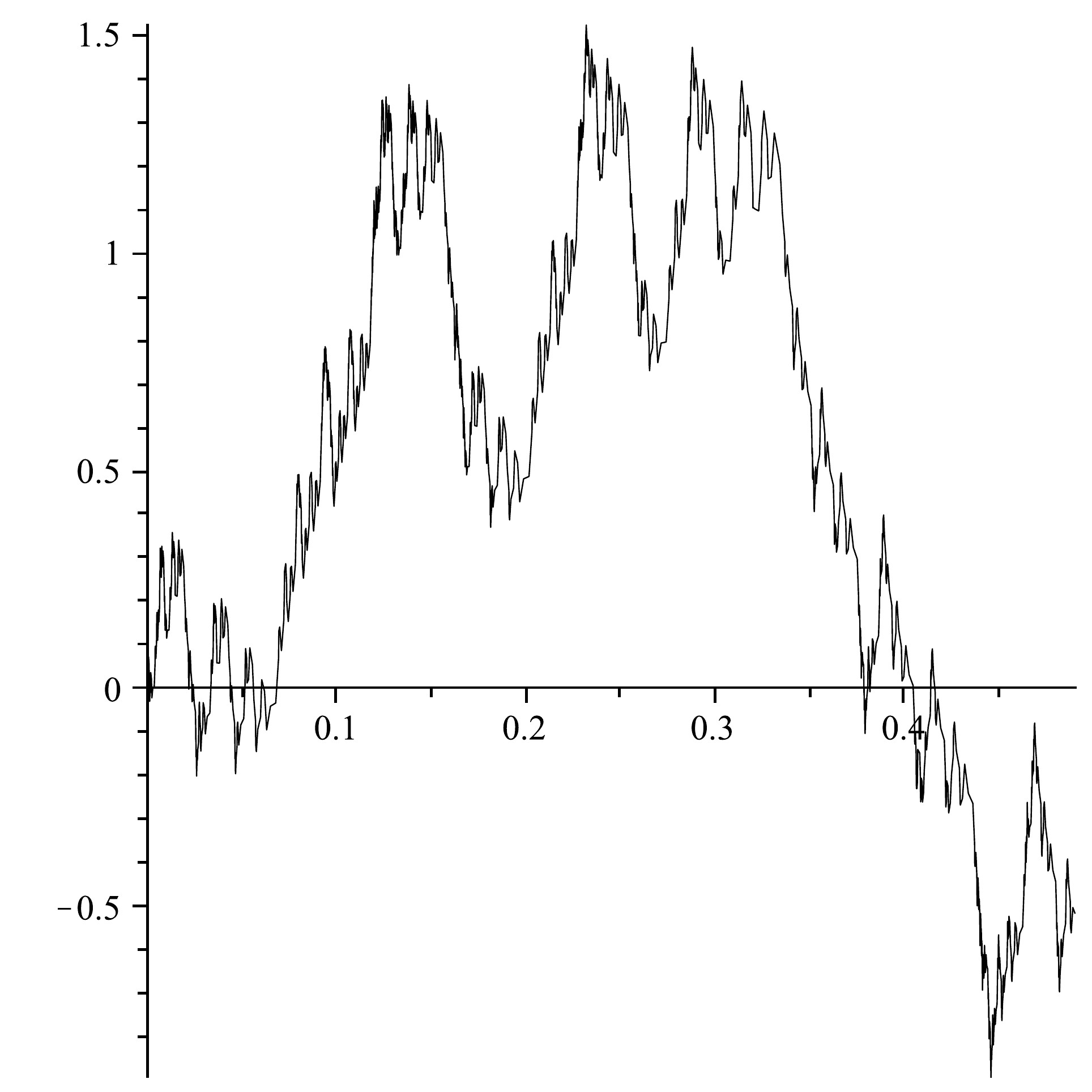}\hfill \hfill
\caption{The driving functions of the van Koch, Sierpinski and Hilbert curves from Fig. \ref{corfig}.}\label{drivefig}
\end{figure}

For $\kappa \geq 0$, chordal SLE$_{\kappa}$ is the random family of
hulls generated by the driving term $\lambda(t) =\sqrt{\kappa} B_t$, where $B_t$ is standard Brownian motion. 
For SLE, it is possible to define an almost surely continuous path $\gamma: [0,\infty) \rightarrow \overline{\mathbb{H}}$, called the trace, so that the
hull $K_t$
 generated by
$\lambda(t)
= \sqrt\kappa B_t$ is the curve $\gamma[0,t]$ filled in.  More precisely, $K_t$ is the complement of the unbounded component of $\mathbb{H} \setminus
\gamma[0,t]$.  See \cite{RS} and, for the case
$\kappa=8$, \cite{LSW14}. Because $r B_{t/r^2}$ has the same distribution as $B_t,$ the law on Loewner traces
induced by SLE is invariant under scaling. Thus it is not very surprising that the deterministic and the
stochastic Loewner equation exhibit very similar phenomena.

The following classification for the SLE$_{\kappa}$ trace was shown in \cite{RS}:
\begin{enumerate}
\item[] For $\kappa \in [0,4]$, $\gamma (t)$ is almost surely a simple
path contained in $\mathbb{H} \cup \{0\}$.
\item[]  For $\kappa \in (4,8)$, $\gamma (t)$ is almost surely a
non-simple path.
\item[] For $\kappa \in [8,\infty)$, $\gamma(t)$ is almost surely a
space-filling curve.
\end{enumerate}
In the deterministic case,  Theorem \ref{LMRthm} and Theorem \ref{2ndphase} give a similar picture.
There are some differences, though: As mentioned above, the existence of a (continuous) trace is no longer 
guaranteed if the norm exceeds 4. Even when assuming the existence of the trace, Lip(1/2) norm $>4$ does not guarantee that the path self-intersects
(for instance, for each $k,$ the trace of $k \sqrt{t}$ is a straight line).
And finally, for $\kappa<4$ the SLE traces have the important property of being uniquely determined by their welding homeomorphism (this easily follows from the H\"older property of the domain $\H\setminus K_t$ \cite{RS} together with 
the Jones-Smirnov removability theorem \cite{JS}). This property is shared by traces of Lip(1/2) norm $<4$ (because quasislits
are conformally removable), but our last example of Corollary \ref{expls} shows that this is no longer true if the norm
is $>4$ (see the discussion at the end of Section \ref{s:expls}).

\section{A Second Phase Transition for Lip$(1/2)$ Driving Terms}\label{phase}

In this section we prove Theorem \ref{2ndphase}, after first considering an illuminating example that is nearly 
(but not quite) a counter-example to the theorem.

\subsection{An example with dense image and small norm}

Let $P=\{z_1, \, z_2, \,z_3, \cdots \}$ be a countable collection of distinct points in $\H$.
We will construct a Lip$(1/2)$ driving function with norm at most 4 that generates a curve which passes through the points in $P$.

We begin by showing that given  $x \in \R$ and $z \in \H$, 
there exists a driving function with  Lip$(1/2)$ norm at most 4
that generates a simple curve from $x$ to $z$ in $\H \cup \{x\}$. 
If Re$(z) = x$, then the constant driving function $\lambda(t) \equiv x$ 
(defined on an appropriate time interval) 
will generate a vertical line connecting $x$ to $z$.  
If Re$(z) \neq x$, then we obtain the desired curve by shifting, scaling and reflecting
the driving term $\l(t)=4\sqrt{1-t}$ appropriately (and again choosing the appropriate time interval), as illustrated in Figure \ref{curvethruZ}.
To see why this will work, note that  the curve generated by $4\sqrt{1-t}$ 
 is a simple curve from 4 to 2 in $\overline{\H}$, 
 and for $\pi/2 < \theta < \pi,$ each ray $\{4+re^{i\theta} \, : \,   r>0\}$ intersects the curve in exactly one point.  
See \cite{KNK} or \cite{LMR}.

\begin{figure}
\centering
\includegraphics[scale=0.8]{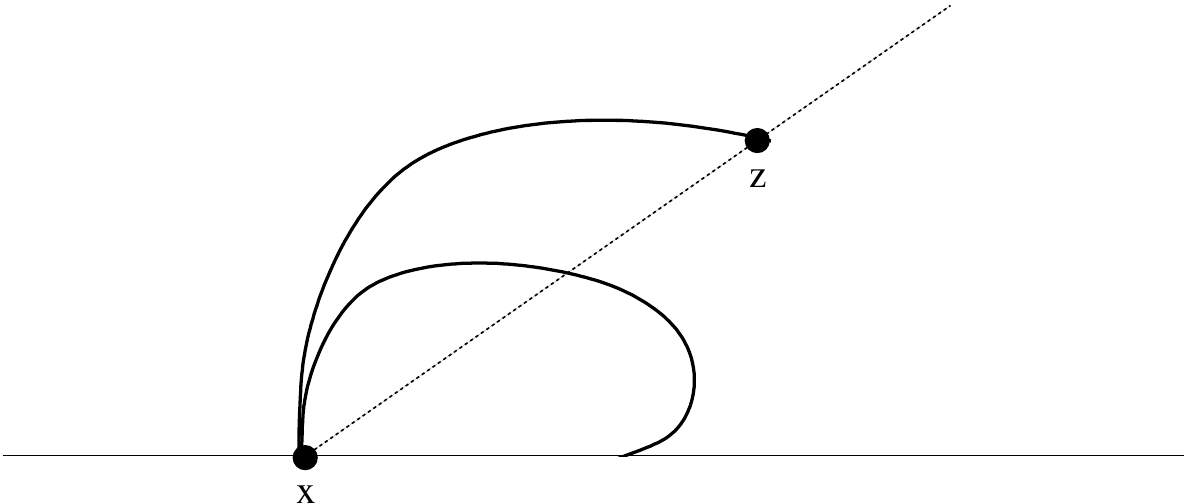}
\caption{A curve from $x$ to $z$ obtained by scaling  the trace driven by $\l(t) = x+4-4\sqrt{1-t}$.}\label{curvethruZ}
\end{figure}

Next, we inductively construct a driving function $\l_n:[0,T_n]\to\R$ with $||\l_n||_{1/2}\leq 4$
in such a way that the associated trace contains all points $z_1,...,z_n$ and so that
$\l_n$ restricted to $[0,T_{n-1}]$ is $\l_{n-1}.$ 
To begin, use the above construction to obtain $\l_1$ as driving function that generates a curve from 0 to $z_1$. 

Now assume that $\l_n$ is already defined.
If $z_{n+1}$ is already in the trace, there is nothing to do, and we set $T_{n+1}=T_n.$
Otherwise we need to append a curve joining the tip to $z_{n+1}$ without increasing the norm.
This is achieved by first setting $\l_{n+1}(t)\equiv \l_n(T_n)$ for $t\in [T_n,\, T_n+\tau_n]$,
 where $\tau_n$  will be determined shortly. 
And secondly, given $\tau_n$, we use the above construction
to obtain a driving term $\hat{\l}(t)$ on some interval $[0,\sigma_n]$ that generates a curve from 
$x_n:=\l_n(T_n)$ to $w_n:=g_{T_n+\tau_n}(z_{n+1})$. 
We then define $T_{n+1}=T_n+\tau_n+\sigma_n$ and
$\l_{n+1}(t)\equiv \hat{\l}(t-(T_n+\tau_n))$ for $t\in [T_n+\tau_n,\,T_{n+1}].$

It remains to show that $\tau_n$ can be chosen so that the Lip$(1/2)$ norm of $\l_{n+1}$ is still at    most 4.
Notice that
$$w_n=g_{T_n+\tau_n}(z_{n+1})= \sqrt{4\tau_n+(g_{T_n}(z_{n+1}) - \l_n(T_n))^2}+\l_n(T_n),$$
since the solution to the Loewner equation driven by the constant $\lambda\equiv0$ is $\sqrt{4t+z^2}.$
By the scaling property there exists $C$ so that 
\be  \label{t3}
\sigma_n \leq C \cdot |w_n - x_n|^2 = C \cdot \left( 4\tau_n+(g_{T_n}(z_{n+1}) - \l_n(T_n))^2 \right).
\ee
In order to guarantee that $\l$ has Lip$(1/2)$ norm at most 4 on $[0, \, T_{n+1}]$, it is enough to require that
\bes
\sqrt{T_n}+\sqrt{\sigma_n} \leq \sqrt{T_{n+1}} = \sqrt{T_n+\tau_n+\sigma_n},
\ees
or equivalently 
\bes 
4T_n \sigma_n \leq \tau_n^2.
\ees
By \eqref{t3} this can easily be accomplished by choosing $\tau_n$ large enough, thus finishing the inductive step of the construction.

\bigskip
\noindent
{\bf Remark.} A slight modification of our construction allows us to arrange that $z_1,z_2,...$ will be visited in this order: In case that a point $z_k$ with $k>n+1$ is contained in the curve from $z_n$ to $z_{n+1}$, 
we must adjust the construction of the driving
term on $[T_n, \, T_{n+1}]$.  
If $z_k$ is contained in the curve by time $T_n+\tau_n$, 
replace the constant driving term  on $[T_n, \,T_n + \tau_n]$ by a Lip$(1/2)$ driving term that is close to a constant but allows the generated curve to avoid all the points $z_i$ for $i>n$.  
This is possible since there is an uncountable family of disjoint curves, with each curve generated by a Lip$(1/2)$ driving term that is close to a constant. 
We make a similar modification if $z_k$ is contained in the curve after time $T_n+\tau_n$.  
There is enough flexibility in the construction that these modifications can be made without increasing the Lip$(1/2)$ norm.

\subsection{How the Loewner equation captures points in $\R$}

In preparation for the proof of Theorem \ref{2ndphase}, 
we will investigate how the Loewner equation \eqref{le} captures real points.
 For a point $x \in\R \setminus \{\lambda(0)\}$, we say $x$ is captured (or killed) by $\l$ 
  if  there exists $t$ so that $g_t(x) = \lambda(t)$, where $g_t$ is generated by $\l$.  
This will occur, for instance, when the trace hits back on the real line, as in Figure \ref{capture}.
We give this event the name ``capture" because $g_t(x)$ attempts to flee from $\lambda(t)$ (that is, the direction of its movement is in the opposite direction from $\l$), and the smaller the distance between  $g_t(x)$ and $\lambda(t)$, the faster   $g_t(x)$ will move to attempt escape.
We first normalize the situation by assuming the capture takes place at time $t=1.$

\begin{figure}
\centering
\includegraphics[scale=0.8]{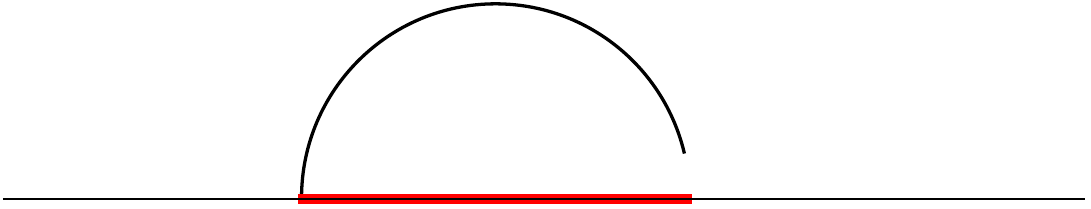}
\caption{When $t=1$ the trace driven by $\l(t) = -3\sqrt{2} \sqrt{1-t}$ will hit back on $\R$ and  each point in the red interval will be captured.}\label{capture}
\end{figure}

Assume that $\l$ is defined on $[0,1]$,  $\,\l(1)=0,$ $\, \l(t) >0$ for $t \in [0,1)$,  and  $g_t$ is generated by $\l$.
Momentarily we will consider the the situation of a point $x < \l(0)$ captured by $\l$ at $t=1$.
First, however, we discuss
a time-changed version of the Loewner equation,
which was introduced in \cite{LMR}.
Set
 \bes
 s=-\ln(1-t), \;\;\; \text{or equivalently,}  \;\;\;t=1-e^{-s}, 
\ees
and define
\bes
G_s(z) := e^{s/2} \, g_{1-e^{-s}}(z) = \frac{g_{t}(z)}{  \sqrt{1-t}} \;\;\; \text{and} \;\;\;
\sigma(s) := e^{s/2}\l(1-e^{-s}) = \frac{\l(t) } {\sqrt{1-t}}.
\ees
Note that $\s$ is defined on $[0, \infty)$, 
and   if $\l \in$ Lip$(1/2)$ with $\norm{\l}_{1/2} \leq C$, then $\s \leq C$ for all $s \in [0, \infty)$.
 By \eqref{le},
\be\label{LDE-G}
\frac{\partial}{\partial s} G_s=
 \frac2{G_s-\s(s)} + \frac{G_s}{2} \;\;\; , \;\;\;G_0(z)=z
\ee
and we say that $G_s$ is generated by $\sigma$. 
 
For $x\in \R \setminus \{ \sigma(0) \}$,  let $x_s = G_s(x)$ be the solution to \eqref{LDE-G}.  
If $x$ is not captured by $\l(t)$ before time $t=1$ (that is, if $\l(t)\neq g_t(x)$ for all $t<1$),
then $x_s=g_t(x)/\sqrt{1-t}$ will exist for all $s \in [0, \infty)$.   
Rewrite \eqref{LDE-G} as
\be \label{gde2}
\frac{\partial}{\partial s} x_s=
 -\frac{1}{2}\frac{x_s^2-\s(s) \, x_s +4}{\s(s)-x_s}.
\ee
When $\s(s) <4$, then the numerator of \eqref{gde2} is always positive, 
implying that $x_s$ is decreasing for $x_s < \s(s)$.

When $\s(s) \geq 4$, we can factor the numerator of \eqref{gde2} to obtain
\be \label{gde3}
\frac{\partial}{\partial s} x_s=
 -\frac{1}{2}\frac{(x_s-A_s)(x_s-B_s)}{\s(s)-x_s},
\ee
where
\bes
A_s := \frac{\s(s) + \sqrt{\s^2(s) -16}}{2} \;\;\; \text{and} \;\;\;  B_s := \frac{\s(s) - \sqrt{\s^2(s) -16}}{2}.
\ees
Now  \eqref{gde3} shows that
$x_s$ is decreasing when $x_s$ is in  $(-\infty, B_s)$ or in $(A_s, \s(s))$,
and that $x_s$ is increasing when $x_s$ is in $(B_s, A_s)$ or $(\s(s), \infty)$.
Roughly, we can think of $A_s$ as attracting and $B_s$ as repelling.  See Figure \ref{Sflow}.
In order for $A_s$ and $B_s$ to be defined for all times $s$,  we set $A_s=B_s=2$ whenever $\s(s) < 4$. 

\begin{figure}
\centering
\includegraphics[scale=0.9]{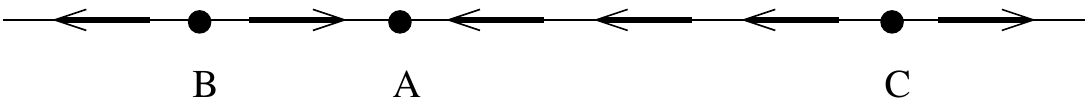}
\caption{The real flow under \eqref{LDE-G} when $\s(s)> 4$: points flow towards $A=A_s$ and away from $B=B_s$ and $C=\sigma(s)$.}\label{Sflow}
\end{figure}
An instructive class of examples is the family of functions $\l(t) = C\sqrt{1-t}$.  Here $\sigma(s) \equiv C $.  In this example, when $C >4$, $A_s \equiv A$ is  an attracting fixed point and $B_s \equiv B$ is a repelling fixed point. 
Every point $x\in[A,C]$ is captured by $\l$ at time $t=1,$
in such a way that $x_s$ decreases to $A$ for all $s>0.$
The next lemma shows that the situation is similar in general: 
Assuming that $x<  \l(0) $ is captured by $\l$ at time $t=1,$
the lemma shows that $x_s$  decreases steadily towards $A_s$ 
until it 
eventually reaches a small neighborhood of the interval $[A_s, \, B_s]$, 
in which it then stays indefinitely.
\begin{lemma}\label{interval}
Let $0<\e<1/2$.
Suppose that
$ \norm{\l}_{1/2} \leq 4+2\e$
and that  $x<  \l(0) $ is captured by $\l$ at time $t=1$ 
with $\; g_1(x)=\l(1)=0$.
There is a finite time $S_0$ and an interval $I$ (containing $[B_s, \, A_s]$) of length $5\sqrt{\e}$ 
so that  $x_s \in I$ for $s \geq S_0$.
\end{lemma}
\begin{proof}
The fact $ \norm{\l}_{1/2} \leq 4+2\e$ implies that $\s(s) \leq 4+2\e$ for all $s$.  Thus,
\bes
A_s \leq 2+ \e + \sqrt{\e(\e+4)} \;\;\; \text{and} \;\;\;
B_s \geq 2+ \e - \sqrt{\e(\e+4)} =: L.
\ees
Set $I=[L, \, L + 5 \sqrt{\e}]$.  
Note that $[B_s, \, A_s]$ will be contained in $I$ for all $s \geq 0$.

First assume that  $x_s > L + 5 \sqrt{\e}$. In order 
to show that $x_s$ decreases towards $I$, recall
\bes
-\frac{\partial}{\partial s} x_s =  \frac{1}{2}\frac{x_s^2-\s(s) \, x_s +4}{\s(s)-x_s}.
\ees
The right hand side is decreasing in $\s$ and increasing in $x_s$,  
so that comparing to $\s_s=4+2\e$ and $x_s=L + 5 \sqrt{\e}$ yields a positive lower bound on 
$-\partial_s \, x_s.$
This proves that there exists $S_0>0$ so that $x_s \leq L+5\sqrt{\e}$ for $s \geq S_0$.

Next, assume that $-1 \leq x_s <L$.
Then 
\bes
-\frac{\partial}{\partial s} x_s 
\geq   \frac{(x_s-L)^2}{12}.
\ees
Thus if there is some $s$ with $-1 \leq x_s <L$, there will be a finite time $S_1$ when $x_{S_1} =-1$. On the other hand,
since $g_t(x)$ decreases to 0 as $t \rightarrow 1$, we must have that $x_s = g_t(x)/\sqrt{1-t} > 0$ for all $s$.
 This contradiction proves that $x_s \in I$ for all $s \geq S_0$.

\end{proof}

Now we will exploit the fact that $x_s$ cannot decrease out of the interval $I$, and 
show that $\sigma$ cannot be bounded above by a constant $M<4$ on a large time interval.

\begin{lemma} \label{tail}
Let $0<\e< 1/2$ and $0< M < 4$.  
Suppose that
$ \norm{\l}_{1/2} \leq 4+2\e$
and that  $x<  \l(0) $ is captured by $\l$ at time $t=1$ 
with $\; g_1(x)=\l(1)=0$.
Let $S_0$  be given as in the previous lemma.
Then there exists $\Delta < \infty$ so that
if $\s(s) < M$ on the time interval $[s_1, \, s_2]$ with $s_1 \geq S_0$, 
then  $s_2-s_1 \leq \Delta.$
In particular, we may take $\Delta = 10 \sqrt{\e} \, (4-M)^{-1}$. 
\end{lemma}

\begin{proof}

Assume that $\s(s) <M$ on the time interval $[s_1, \, s_2]$ with $s_1 \geq S_0$.  From Lemma \ref{interval}, $x_s \in I$ for $s \in [s_1, \, s_2]$.
We also know that $x_s$ will be decreasing on  $[s_1, \, s_2]$, since $\s(s) < 4$.  
We wish to determine the amount of time needed for $x_s$ to decrease from the right endpoint of $I$ to the left endpoint of $I$.   Since the right hand side of
\bes
-\frac{\partial}{\partial s} x_s =  \frac{1}{2}\frac{x_s^2-\s(s) \, x_s +4}{\s(s)-x_s}
\ees
is decreasing in $\s$, the larger the value of $\s$ the longer it will take to exit $I$.  Therefore, we may simply consider the case when $\s(s) \equiv M$.
The right hand side of 
\be \label{Mde}
-\frac{\partial}{\partial s} x_s =  \frac{1}{2}\frac{x_s^2-M  x_s +4}{M-x_s}
\ee
has a minimum when $x_s = M -2$, and so
\bes
-\frac{\partial}{\partial s} x_s \geq \frac{4-M}{2}. 
\ees
 Define
\be \label{del2}
\Delta = \frac {10\sqrt{\e} }{4-M}.
\ee
Since $I$ is an interval of length $5 \sqrt{\e},$ then $x_s$ must exit $I$ after decreasing for a time interval of length $\Delta$.

 \end{proof}

In our last lemma, we simply restate the results of Lemma \ref{interval} and Lemma \ref{tail} without reference to the time change.
This gives a quantitative version of the following fact:  if $M<4$ and $|\lambda(T)-\lambda(t)| \leq M \sqrt{T-t}$ for all $t < T$, then it is not possible for any real point to be captured at time $T$.

\begin{lemma}\label{tail2}
Let $0<\e<1/2$ and $0<M<4$. 
Suppose that $\l$ is a Lip$(1/2)$ driving function with
$ \norm{\l}_{1/2} \leq 4+2\e$.  
Further suppose that $x \in \R \setminus \{\l(0)\}$ is captured at time $T$, meaning $g_T(x)=\l(T)$.
Then there exists $S_0 < \infty $ and $\Delta < \infty $
(with $S_0$ and $\Delta$ depending only on $\epsilon$ and $M$), 
so that whenever $s \geq S_0$, the time interval 
$[(1-e^{-s})T, \, (1-e^{-(s+\Delta)})T]$ contains a time $t$ satisfying
$|\l(T)-\l(t)| \geq M \sqrt{T-t}.$
Further,  we may take $\Delta= 10\sqrt{\e}(4-M)^{-1}$.

\end{lemma}

Note that  if  $x > \l(0)$, then we can conclude that  $\l(T)-\l(t) \geq M \sqrt{T-t}.$

\subsection{Proof of Theorem \ref{2ndphase}}

\begin{proof}[Proof of Theorem \ref{2ndphase}]

Assume that $\l$ is a Lip$(1/2)$ driving function that generates a curve $\gamma$ with non-empty interior.
We would like to show that $\norm{\l}_{1/2} > 4.0001$. To this end, we assume that
$\norm{\l}_{1/2} \leq 4+2\e$ and strive for a contradiction when $\epsilon$ is sufficiently small.

There must be some finite time $T$ so that $\gamma[0,T]$  has non-empty interior: 
If not, then there is some closed disk $D$ in $\mathbb{C}$ that is the countable union of the nowhere-dense sets $D \cap \gamma[0,n]$,
contradicting the\ Baire Category Theorem.
If $\gamma(t_0)$ is an interior point, then $\l(t_0)$ is an interior point of $g_{t_0}(\gamma)$ (with respect to $\overline\H$).
Replacing $\l(t)$ with $\l(t+t_0)-\l(t_0)$ and scaling appropriately,
we may therefore assume that there is an interval $I\subset K_1\cap\R_+$.

Each point $x\in I$ will be captured at a distinct ``capture time''.
Since there are uncountably many points in $I$, there exist capture times $T_1< T_2$ 
so that $T_2 - T_1 \leq e^{-2 S_0}T_2$, where $S_0$ is given as in Lemma \ref{tail2}. (Note that $S_0$ depends on $\e$ and $M \in (0,4)$,  and these will be specified later.)

\begin{figure}[h]
\centering
\includegraphics[scale=0.8]{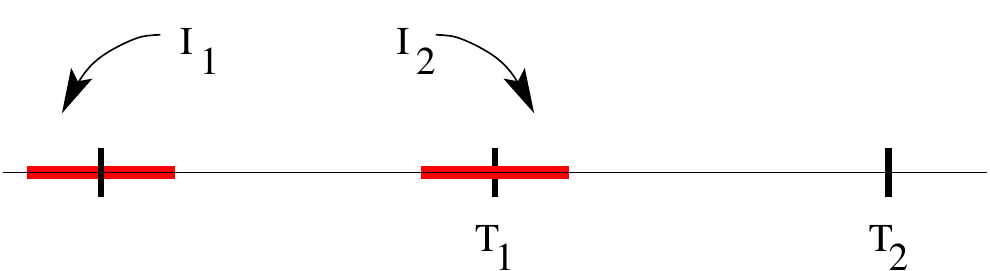}
\caption{The times $T_1$ and $T_2$ and their corresponding intervals $I_1$ and $I_2$.}\label{Tintervals}
\end{figure}

 Let $\Delta= 10\sqrt{\e}\,(4-M)^{-1}$  be  as in Lemma \ref{tail2},
and consider the time interval $I_2 =[(1-e^{-s})T_2, \, (1-e^{-(s+\Delta)})T_2]$, 
 where  $s$ is chosen so that  $T_1$ is the midpoint of this interval.
 Let $I_1$ be an interval of the same length as $I_2$, but shifted to the left by $T_2 -T_1$.  In other words, the distance from $T_1$ to the midpoint of $I_1$ is the same as the distance from  $T_2$ to $T_1$, the midpoint of $I_2$, as shown in Figure \ref{Tintervals}.
 Then by Lemma \ref{tail2}, there exists  $t_2  \in I_2$  and $t_1 \in I_1$ so that 
 \be \label{Testimates}
 \l(T_2)-\l(t_2) \geq M \sqrt{T_2-t_2} \;\;\; \text{ and } \;\;\; \l(T_1)-\l(t_1) \geq M \sqrt{T_1-t_1}.
 \ee
  
We would like to conclude that, for appropriate choices of $M$ and $\epsilon,$
$$\l(T_2)-\l(t_1)>(4+2\e)\sqrt{T_2-t_1}$$
which would yield the desired contradiction to our standing assumption $\norm{\l}_{1/2} \leq 4+2\e$.

\noindent
Since
\begin{align*}
\l(T_2)-\l(t_1) &= \left(  \l(T_2)-\l(t_2) \right) + \left( \l(t_2)-\l(T_1) \right) + \left( \l(T_1) - \l(t_1) \right)\\
&\geq M\sqrt{T_2-t_2} - (4+2 \e) \sqrt{|t_2-T_1|} + M \sqrt{T_1-t_1},
\end{align*}
it suffices to show
\be \label{inequality}
(4+2 \e) \sqrt{T_2-t_1} - M\sqrt{T_2-t_2} + (4+2 \e) \sqrt{|t_2-T_1|} - M \sqrt{T_1-t_1} < 0.
\ee
Set $M=3.5$ and $\e = 0.00005$. 
Then the left hand side of \eqref{inequality} is increasing in $t_1$.
Therefore, we can assume that $t_1$ is the right endpoint of $I_1$.
Then 
\begin{align*}
&T_2 - t_1 =  2e^{-(s+\Delta)}T_2+(1/2) (e^{-s}-e^{-(s+\Delta)})T_2,\\
&T_i - t_i \geq e^{-(s+\Delta)}T_2\, \text{ for  }i=1,2,  \,\, \text{ and  }\  \\
&|t_2-T_1| \leq (1/2) (e^{-s}-e^{-(s+\Delta)})T_2.
\end{align*}
Now computation shows that the left hand side of  \eqref{inequality} is less than $ -0.125 \sqrt{e^{-s} T_2}$
and the theorem is proved.

\end{proof}

\section{A criterion for Lip(1/2) driving terms}\label{s:domains}

In this section we prove Theorem \ref{qdisc} and then discuss applications to several examples.

\subsection{Proof of Theorem \ref{qdisc}}

In order to prove Theorem \ref{qdisc}, we need to make a connection between the geometry of a  family of hulls 
and the  Lip$(1/2)$ norm of the driving term.  
The following simple lemma provides this connection.

\begin{lemma}\label{cap}
Let $K_T$ be the hull generated by the driving term $\lambda$ at time 
$T$.  Then 
$$\max \{ \operatorname{Im}(z) \text{ }| \text{ } z \in K_T \}\leq 2\sqrt{T},$$
and 
$$\abs{\lambda(T)-\lambda(0)} \leq 4 \operatorname{diam}(K_T).$$

\end{lemma}

\begin{proof}

Let $f_t$ be the family of conformal maps 
generated by $\xi_t=\lambda(T-t)$ via the backward Loewner equation
\eqref{ble}, so that 
$f_{T}$ is the conformal map from $\mathbb{H}$ onto $\mathbb{H} \setminus K_T$.

\bigskip\noindent

The first claim just says that $K_t$ cannot grow vertically any faster 
than the vertical slit generated by the constant function.  To see this,
write $f_t=x_t+i y_t$ and notice that $\partial_t y_t = 2y_t/((x_t-\xi_t)^2+y_t^2)\leq 2/y_t,$
or $\partial y_t^2\leq 4.$ Thus $y_t\leq \sqrt{y_0^2+4t}$ and the claim follows by letting $y_0\to0.$

\bigskip\noindent

For the second claim, let $A=g_T(\partial K_T)$, which means that $A$ is an interval in $\mathbb{R}$ and $f_T(A)$ is the boundary of the hull $K_T$ in the upper halfplane. 
Once we know that $\lambda(0)$ and $\lambda(T)$ are contained in A,
the claim is established by the following facts about logarithmic capacity: 
$\operatorname{diam}(A) = 4\operatorname{cap}(A)$  and  
$\operatorname{cap}(A)=\operatorname{cap}(K_T^*) \leq \operatorname{diam}(K_T),$ where $K_T^*$ is the union of $K_T$ and its reflection across $\R$.

Notice that $\lambda(T) = \xi(0) \in A$, since $f_T(\xi(0))$ is the ``tip'' of $K_T$. 
It remains to show that $\lambda(0) = \xi(T)$ is also in $A$.  
For $\epsilon >0$, set $a_1 = \min(A) - \epsilon$, set $a_2 = \max(A)+ \epsilon$, and let $x_i(t)$ be the solution to \eqref{ble} with initial value $a_i$ for $i=1,2$.  Note that $\partial_t \,x_1(t) >0$ and $\partial_t \, x_2(t) < 0$.  Further, 
$x_i(t) \neq \xi(t)$  since $a_i \notin A$.
Thus, $a_1 < x_1(t) < \xi(t) < x_2(t) < a_2$ for all $t \in [0,T]$.  Letting $\epsilon \rightarrow 0$, this
implies that   $\xi(T) \in A$.  

\end{proof}

\begin{figure}
\centering
\includegraphics[scale=0.8]{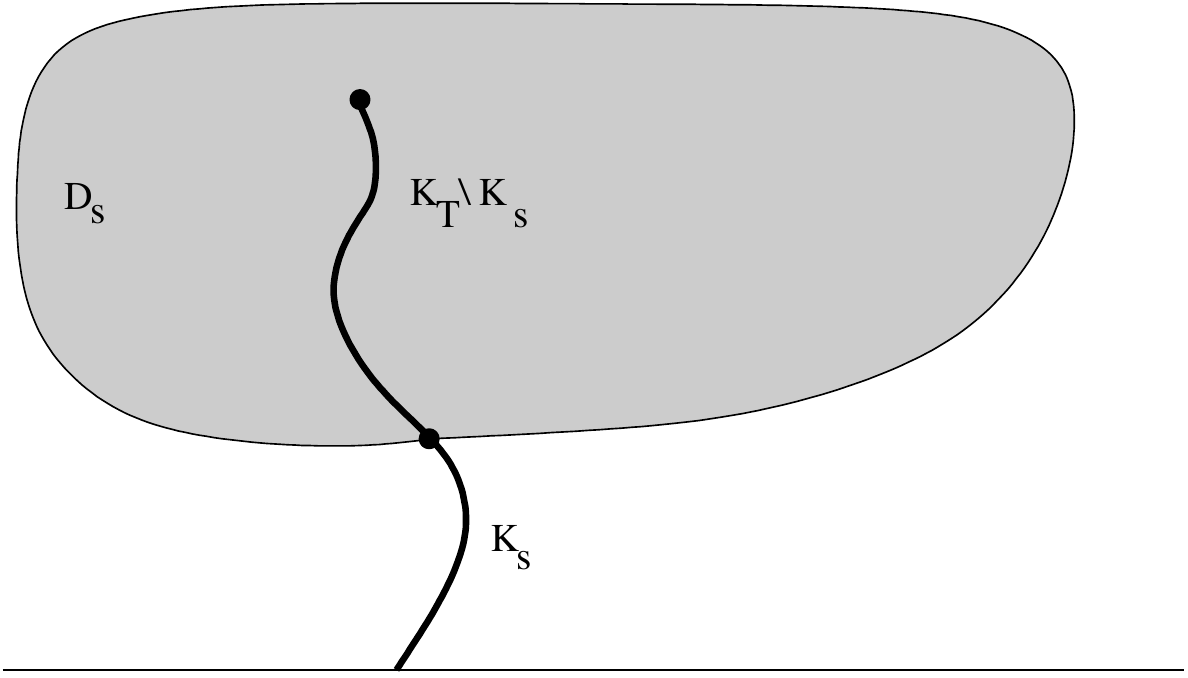}
\caption{The quasidisc $D_s$ of Theorem \ref{qdisc}.}\label{quasidisc}
\end{figure}

We are now ready to prove Theorem \ref{qdisc}.

\begin{proof}[Proof of Theorem \ref{qdisc}]

Let $s,t \in [0,T]$ with $s < t$, and 
set $\hat{K}_{s,t} := g_s \left(  K_t \setminus K_s \right)$. Lemma \ref{cap} implies 
$$\abs{\lambda(t)- \lambda(s)} \leq 4 \text{ diam}(\hat{K}_{s,t}).$$
  We claim that  
\begin{equation}\label{diamheight}
\operatorname{diam}(\hat{K}_{s,t}) \leq C \max \{ \operatorname{Im}(z) \text{ }|\text{ } z \in \hat{K}_{s,t} \}.
\end{equation}
Along with Lemma 
\ref{cap} this gives 
$$\abs{\lambda(t)- \lambda(s)} \leq  4C \max \{ \operatorname{Im}(z) \text{ }|\text{ } z \in \hat{K}_{s,t} \} 
\leq 8C\,\sqrt{t-s}.$$
It therefore remains to prove \eqref{diamheight}.

By assumption, there is a $k$-quasi-disc $D_s \subset \mathbb{H}$, 
with $K_s$ in the complement of $D_s$.  
Therefore, $g_s$ is conformal on $D_s$.
Thus there is a quasi-conformal map from $\hat{\mathbb{C}}$ to 
itself that agrees with $g_s$ on $D_s$ (see Section I.6 of \cite{Le} for 
one possible reference), and the quasi-conformal constant of this map depends only on 
$k$.  Hence $g_s$ is quasi-symmetric on 
$\overline{D_s}$, 
with constant depending only on $k$.  Recall that a homeomorphism $g$ is 
quasi-symmetric if $\abs{z-z_0} \leq a \abs{w-z_0}$ implies that  
$\abs{g(z)-g(z_0)} \leq c(a) \abs{g(w)-g(z_0)}.$

Let $z_0$  be a point in $K_t \setminus K_s$ that maximizes dist$(z, 
\partial D_s)$.  Let $z$ 
be in $K_t \setminus K_s$, and let $w$ be in $\partial D_s$.
Using property (3), we have that   
$$\abs{z-z_0} \leq \text{diam}\left( K_t 
\setminus K_s \right) \leq C_0 \text{ dist}(z_0, \partial D_s) \leq
C_0\, \abs{w- z_0}.$$ 
By quasi-symmetry, 
$$\abs{g_s(z)- g_s(z_0)} \leq C \,\abs{g_s(w)- g_s(z_0)},$$ 
where the 
constant $C$ depends only on $C_0$ 
and $k$. Maximizing over $z$ and minimizing over $w$ establishes 
\eqref{diamheight}, completing the proof.
\end{proof}

\subsection{Examples}\label{s:expls}

We illustrate the use of Theorem \ref{qdisc} by applying it to each of the curves mentioned in Corollary \ref{expls}.
To show that each of these examples are generated by a Lip$(1/2)$ driving term,  we must construct the family of $k$-quasi-discs $D_s$ required in the hypotheses of Theorem \ref{qdisc}.

\bigskip

\noindent
{\bf The van Koch curve.}

\bigskip
\noindent
It was already shown in \cite{MR} that this curve, as well as any other quasi-slit, is driven by a Lip(1/2) function.
To obtain a proof based on Theorem \ref{qdisc}, let $F$ be a quasi-conformal automorphism of $\H$
fixing $\infty$, and for $0<\tau<1$ let $\Delta_{\tau}$ be the triangle with vertices $-\tau,\tau, i\tau.$
Then the domain $D_{\tau}=F(\H\setminus \Delta_{\tau})$ is a quasi-disc for each $\tau$, and it easily
follows from the quasi-symmetry of $F$ that the curve $F[0,i]$ satisfies the assumptions of  Theorem \ref{qdisc}.

\bigskip

\noindent
{\bf The Hilbert space filling curve.} 

\bigskip

\begin{figure}
\centering
\includegraphics[scale=0.35]{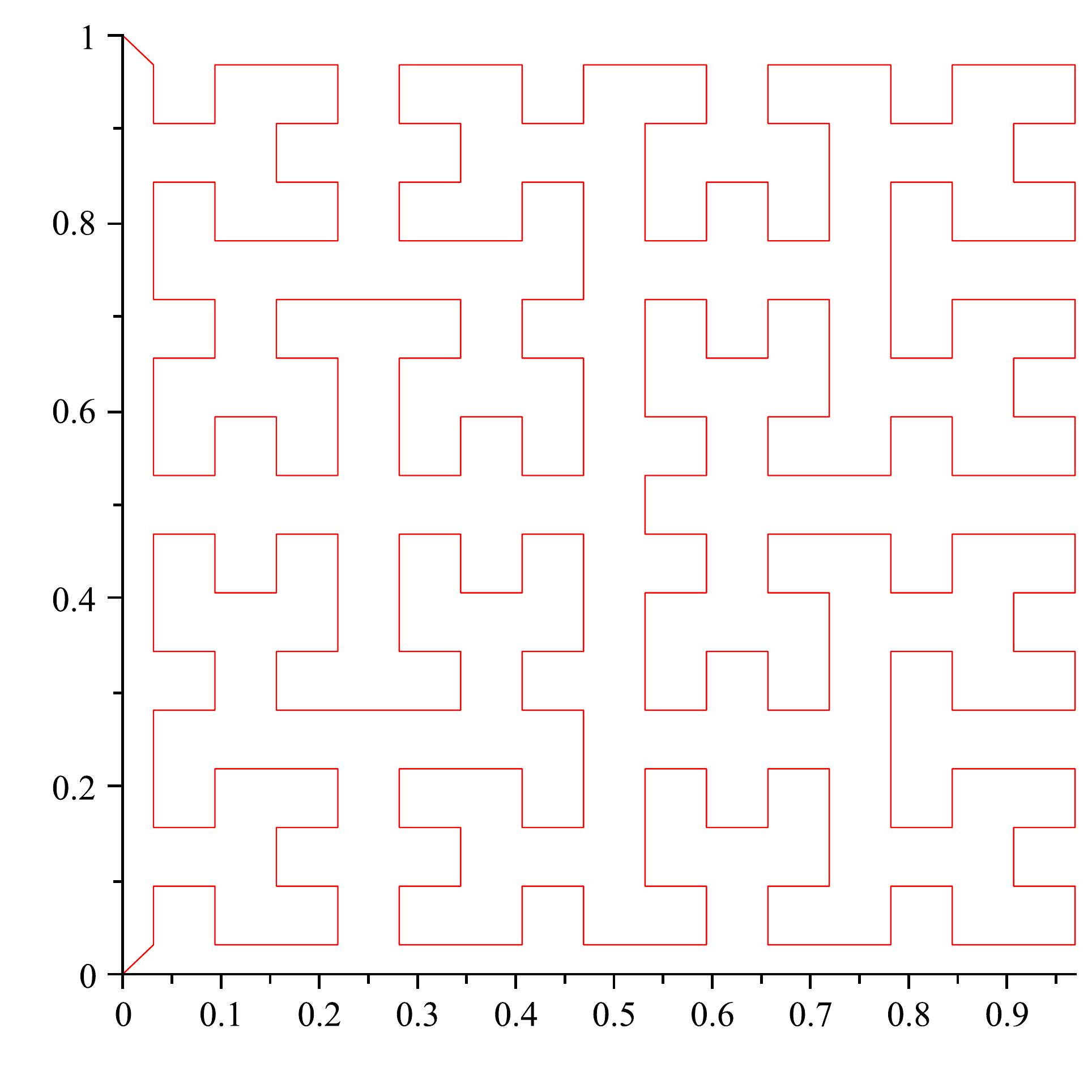}
\includegraphics[scale=0.35]{space6}
\caption{Two curves approximating the space-filling Hilbert curve. 
}\label{space}
\end{figure}

\noindent
In this case, setting $D_s := \mathbb{H} \setminus \Gamma[0,s]$ will do.  
To see this, it is sufficient to show that both $D_s$ and its complement are
John domains (see \cite{P}, Chapter 5). It suffices to show that the interior 
of $\Gamma[0,s]$ is John, since $D_s$ and its complement have the 
same geometry (because $\Gamma[0,s]\cup \Gamma[s,1]=[0,1] \times [0,1]$).

To see that the interior $G$ of $\Gamma[0,s]$ is John, we would like to show that there exists $C>0$ so that for every
rectilinear crosscut $[a,b]$ of $G$, the diameter of one of the two components of
$G\setminus[a,b]$ is bounded above by $C|b-a|.$ 
Let $[a,b]$ be a rectilinear crosscut of $G$, 
and  let   $A$ denote the component of $G\setminus[a,b]$ with smaller diameter.
It is possible to do a case study in which we consider all the possible configurations of $A$ and $G\setminus A$.  However, it will be much simpler to consider only ``small" crosscuts, and so we assume, as we may, that $|b-a| < (1/100) \text{ diam}(G)$.
With this assumption, there are just two cases, which are roughly described as (1)  $A$ is in a ``corner" of $\Gamma[0,s]$, or (2) $A$ is contained in the ``end" of $\Gamma[0,s]$.

Case (1): $A$ is contained in a right triangle with hypothenuse $[a,b]$.  In this case it is clear that  $ \text{diam}(A) = |b-a|$.

Case (2):  Suppose $A$ is not contained in a right triangle with hypothenuse $[a,b]$.  Then $A$ must be near the ``end" of $\Gamma[0,s]$.  
We will describe this carefully, but first
let us pause for a minute and remember how the Hilbert space-filling curve grows: for any integer $k$, we can decompose the unit square into $2^{2k}$ squares with disjoint interior and sidelength $2^{-k}$.  
 The Hilbert space-filling curve completely fills out each square before venturing into the interior of another square, and it never returns to the interior of a square after leaving it.  
Now let $m$ be the unique integer so that $2^{-(m+1)} \leq |b-a| < 2^{-m}$.  
Then there are a finite number of squares  $S_1, S_2, \cdots, S_N$ 
with disjoint interior and sidelength equal to $2^{-m}$ 
so that  $\Gamma[0,s] = \left( \cup_{k=1}^{N-1} S_k \right) \cup \left(\Gamma[0,s] \cap S_N\right)$.  
(That is, $\Gamma[0,s]$ fills out the first $N-1$ squares, but might not completely fill out the last square $S_N$.)
Further, assume that the squares are listed  in the order in which they are visited by the Hilbert space-filling curve.  
 Because of the choice of $m$, the sidelength of the squares $S_k$ is greater than $|b-a|$.  If we are not in case (1),  it follows that  $A$ must be contained in $S_{N-1} \cup S_N$, which is a $2^{-m}$ by $2 \cdot 2^{-m}$ rectangle. 
 Hence, diam$(A) \leq 2^{-m}\sqrt{5} \leq 2\sqrt{5}\,  |b-a|$.

\bigskip

\noindent
{\bf The half-Sierpinski triangle.} 

\begin{figure}
\centering
\includegraphics[scale=0.35]{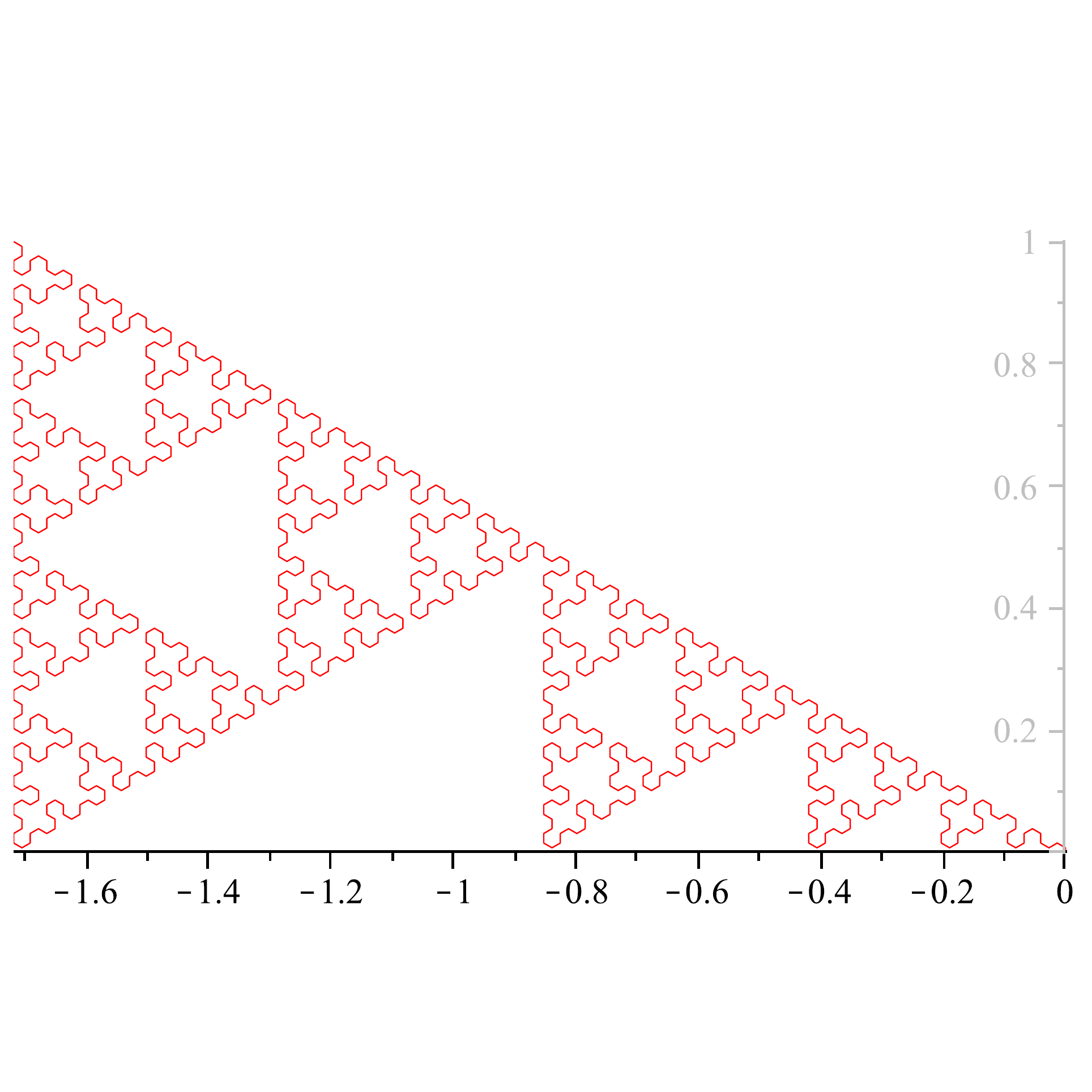}
\includegraphics[scale=0.35]{hs8}
\caption{Two curves approximating a curve that traces out half of a Sierpinski triangle. 
}\label{halfsierp}
\end{figure}

\bigskip
\noindent
Here, setting $D_s = \mathbb{H} \setminus \Gamma[0,s]$ will almost work but needs minor modifications. 
First, we have to take the unbounded component of $\mathbb{H} \setminus \Gamma[0,s]$ (that is, 
we just fill in the holes (white triangles) of $\mathbb{H} \setminus \Gamma[0,s]$). 
And second, we add all those white triangles (to the complement of $D_s$)
for which one of its edges is contained in $\Gamma[0,s]$.
For instance, consider the point $p = -\sqrt{3}/2 + i/2$ and the time $s_p$
for which $\Gamma(s_p)=p.$ For each $s>s_p,$ the point $p$ is a cut-point of $\Gamma[0,s]$. For all those
$s$, the largest white triangle belongs to the complement of $D_s$. 
As in the previous example, 
it is possible to show
that $D_s$ satisfies the requirements of Theorem \ref{qdisc}.

\bigskip

\noindent
{\bf A curve of positive area.}

\bigskip
\noindent
Our final example is a standard construction of a curve with positive area.
It is interesting because a well-known construction based on the Beltrami equation
shows that every set of
positive area admits non-trivial homeomorphisms that are conformal off that set.
In particular, we obtain that the conformal welding homeomorphism 
of a Lip(1/2) driven curve does not neccessarily determine the curve uniquely.  
See \cite{AJKS} for a discussion of this in the context of random weldings.

\bigskip

The curve $\beta:[0,1]\to[0,1]\times [0,1]$  will be constructed in stages.  
We begin by defining the pieces of the curve which lies in $[0,1] \times [0,1] \setminus S_1$, 
where $S_1$ is the disjoint union of 4 closed squares 
which together have total area $1-\epsilon_1$, for $\epsilon_1 \in [0,1)$.  
The definition of $\beta$ in $[0,1] \times [0,1] \setminus S_1$ is shown in Figure \ref{bblock}.

\begin{figure}[h]
\centering
\includegraphics[scale=0.5]{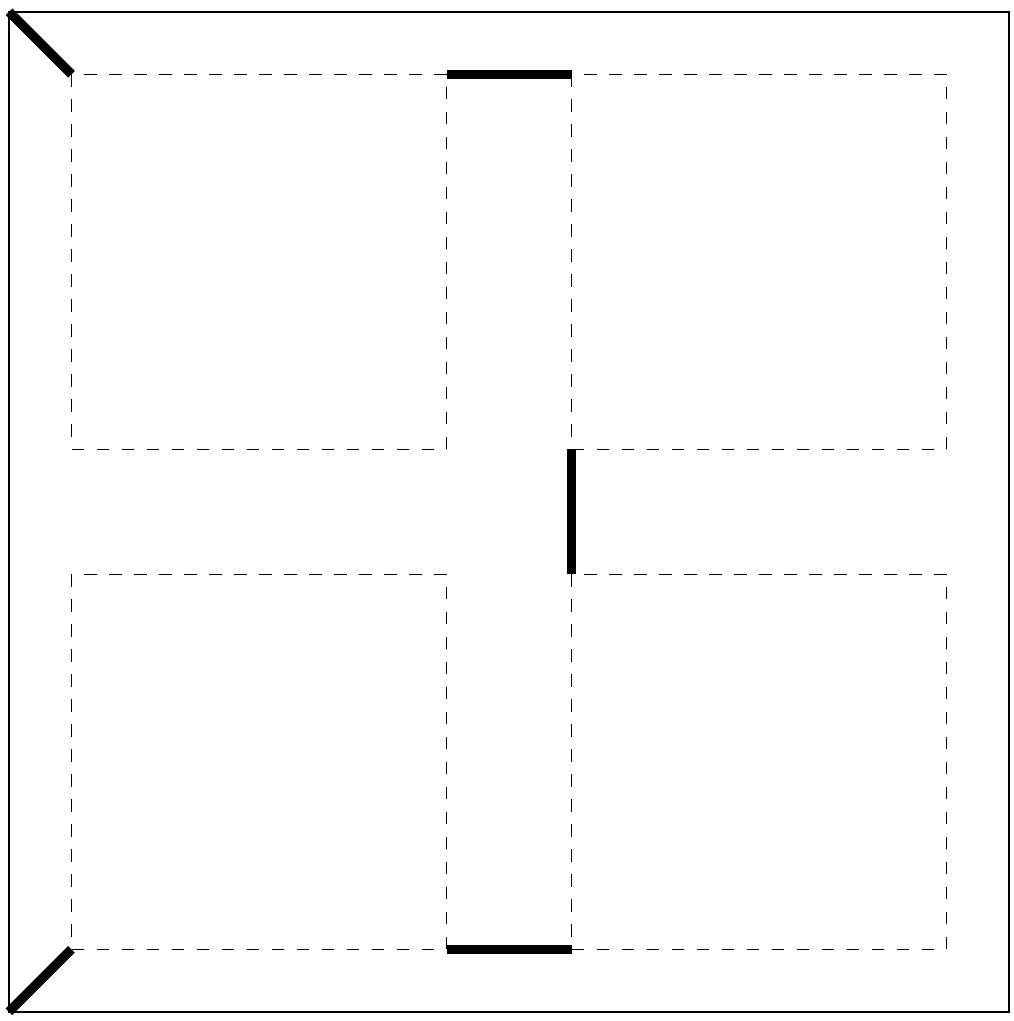}
\hspace{.5in}
\includegraphics[scale=0.5]{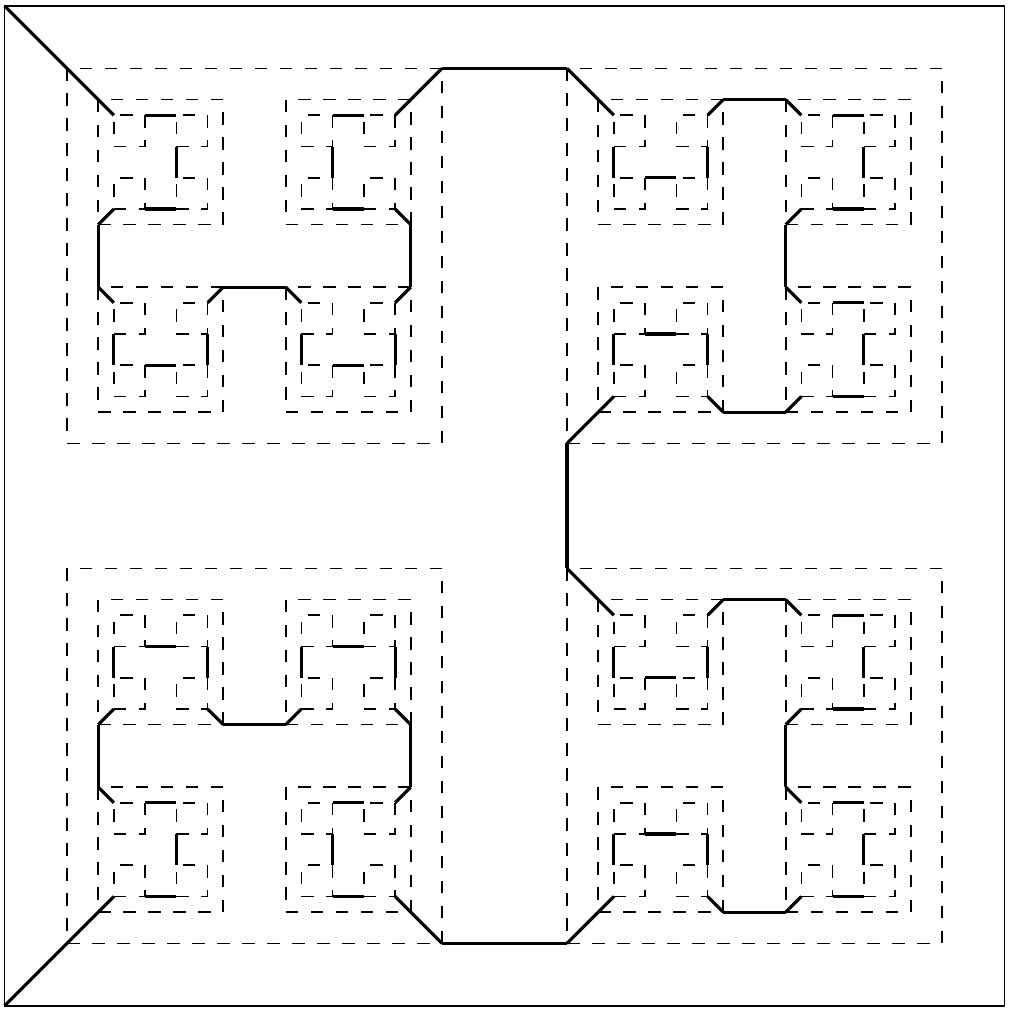}
\caption{The pieces of the curve $\beta$ which lie in $[0,1] \times [0,1] \setminus S_1$ respectively $[0,1] \times [0,1] \setminus S_3$ are shown as thick solid lines. On the left, $S_1$ is the set bordered by dashed lines.}\label{bblock}
\end{figure}

Fix numbers $0\leq \epsilon_k<1.$
We take $S_k$ to be the disjoint union of $4^k$ equally-sized, closed squares which are contained in $S_{k-1}$ and which have area$(S_k)= (1-\epsilon_k) \text{ area}(S_{k-1})$.  Roughly, we define $\beta$ in $S_{k-1} \setminus S_{k}$ by fitting the left image of Figure \ref{bblock} in each of the $4^{k-1}$ squares in $S_{k-1}$, after scaling and rotating the picture appropriately, taking care to match corners so that our final curve will be continuous.  We show $\beta \setminus S_3 $ in Figure \ref{bblock}.  The limiting object is the curve $\beta$. Notice that the Hilbert curve is the special case 
$\epsilon_k=0$ for all $k,$ whereas $\epsilon_k\geq\epsilon>0$ generates a quasislit.

To create a curve with positive area, choose $\epsilon_k \in (0,1)$ so that $\sum_{k=1}^{\infty} \epsilon_k < \infty$.  Note that  each point in $\bigcap_{k=1}^{\infty} S_k$ must lie on the curve,
and
$$\text{area} \big( \bigcap_{k=1}^{\infty} S_k \big) = \prod_{k=1}^{\infty} (1-\epsilon_k) >0.$$

The decreasing sequence of domains $D_s=\H\setminus A_s$ needed for Theorem \ref{qdisc} 
can be constructed by combining the ideas behind the constructions in the previous examples. 
A rough description is as follows. 
Denote $S_{11}=[a,b]\times[a,b]$ the first (lower left) square of $S_1,$ 
and $S_{12}=[c,d]\times[a,b]$ the second (lower right).  Squares $S_{13}$ and $S_{14}$ are the top right and top left squares of $S_1$.
Set $\beta(s)=x_s + i\, y_s$, and notice  that $x_0 = 0 = y_0$ and  $x_s = y_s$ for  small $s$.
If $\beta[0,s]\cap S_{11}=\emptyset$, 
     take $A_s$ to be the square $[0, x_s] \times [0, y_s]$.
If $\beta[0, s]\cap S_{11}=\beta[0,1]\cap S_{11}$ and $\beta[0,s]\cap S_{12}=\emptyset$, 
     set $A_s=[0, x_s]\times[0, b]$. 
If $\beta[0,s]\cap S_{12}=\beta[0,1]\cap S_{12}$ and $\beta[0,s]\cap S_{13}=\emptyset$, 
     set $A_s=[0,1]\times[0, y_s].$ 
Proceed similarly for $S_{13}$ and $S_{14}$,
and use the self-similar nature to extend the definition to higher levels. 
Then $D_s$ are quasi-discs as can be seen
by arguments similar to those for the Hilbert curve. 
We leave the details to the reader.

\subsection{The Sierpinski Arrowhead Curve}

In Figure  \ref{sierp}, we give pictures of curves approximating the Sierpinski arrowhead curve, which traces out a full Sierpinski triangle.  Note that our proof of Theorem \ref{2ndphase} applies to this example (although not to our half-Sierpinski example) since there is an interval along the real line for which each point is killed at a distinct time.

\begin{figure}[h]
\centering
\includegraphics[scale=0.4]{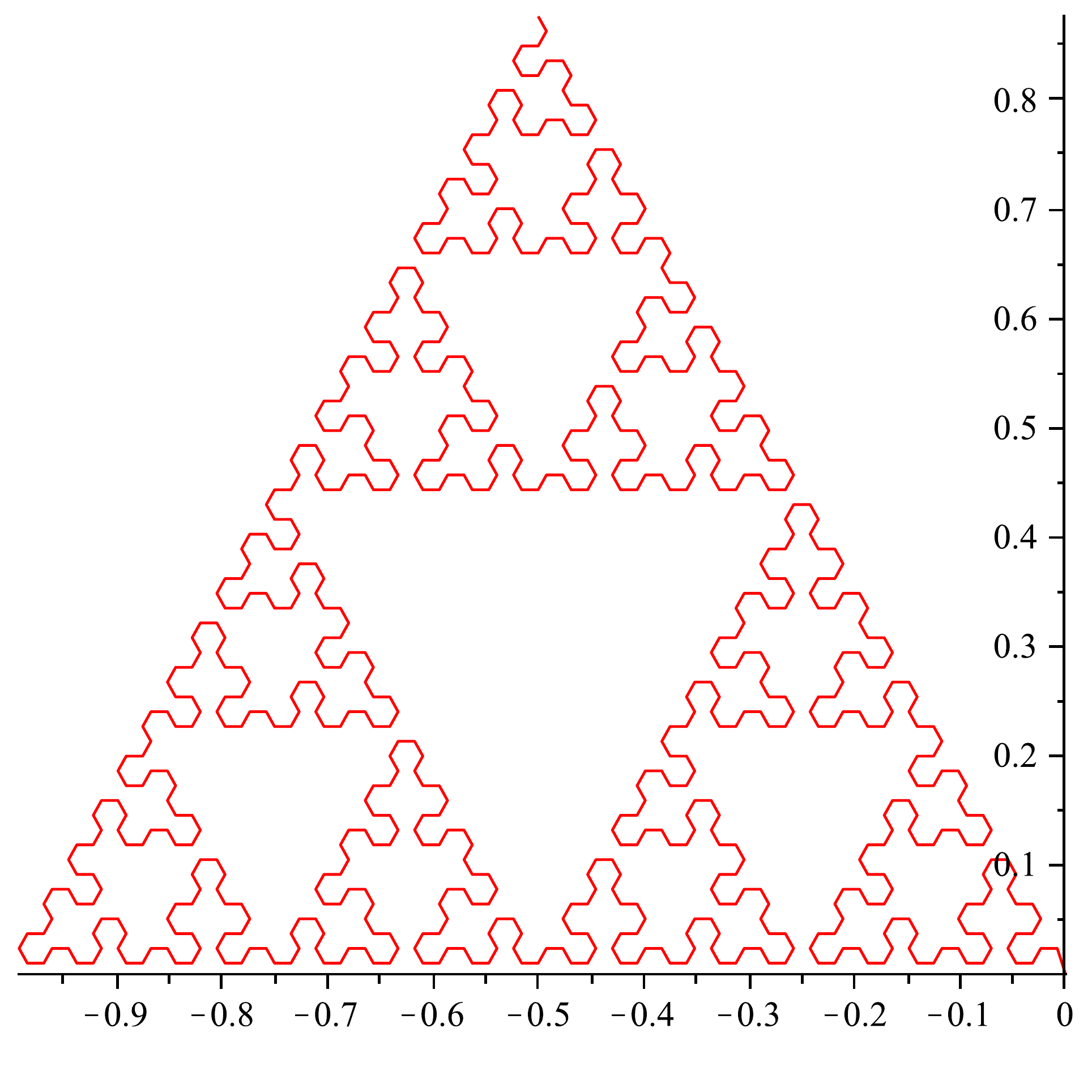}
\includegraphics[scale=0.4]{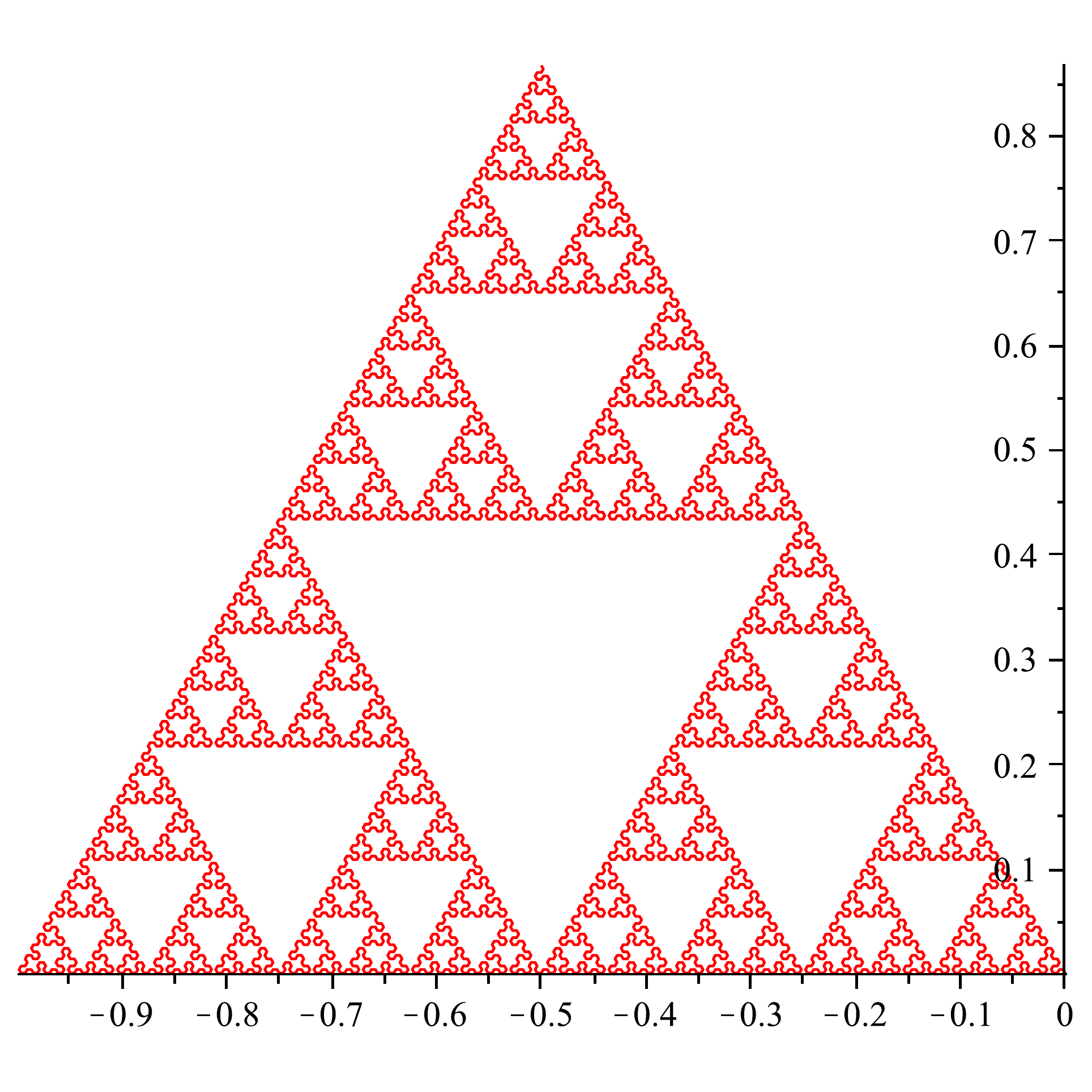}
\caption{Two curves approximating the Sierpinski arrowhead curve. 
}\label{sierp}
\end{figure}

\end{doublespace}


\begin{thebibliography}{abc}



 \bibitem[AJKS]{AJKS} K. Astala, P. Jones, A. Kupiainen, E. Saksman,
Random conformal weldings, preprint.

\bibitem[JS]{JS} P. Jones, S. Smirnov, 
Removability theorems for Sobolev functions and QC maps, 
{\it Ark. Mat. \bf 38} (2000), no. 2, 263--279.

\bibitem[KNK]{KNK} W. Kager, B. Nienhuis, L. Kadanoff, 
Exact solutions for Loewner evolutions,
{\it J. Stat. Phys. \bf 115} (2004), 805--822.

\bibitem[La]{La} G. Lawler, Conformally Invariant Processes in the Plane,
{\it Mathematical Surveys and Monographs, \bf 114}. 
American Mathematical Society, Providence, RI, 2005.

\bibitem[LSW]{LSW14}  G. Lawler, O. Schramm, W. Werner,
Conformal invariance of planar loop-erased random walks and uniform spanning trees,
{\it Ann. Probab.\bf 32} (2004), 939--995.

\bibitem[Le]{Le} O. Lehto, Univalent functions and Teichm{\"u}ller spaces,
	Springer-Verlag, 1987.

\bibitem[Li]{L} J. Lind, A sharp condition for the Loewner equation to generate slits,
{\it Ann. Acad. Sci. Fenn. Math. \bf 30} (2005), 143--158.

\bibitem[LMR]{LMR} J. Lind, D.E. Marshall, S. Rohde, Collisions and Spirals of Loewner Traces,
{\it Duke Math. J. \bf 154} (2010), 527--573. 

\bibitem[MR]{MR} D.E. Marshall, S. Rohde, The Loewner differential equation and slit mappings,
{\it J. Amer. Math. Soc. \bf 18} (2005),  763--778. 



\bibitem[P]{P} C. Pommerenke, Boundary behaviour of conformal maps, Springer, 1991.

\bibitem[RS]{RS} S. Rohde, O. Schramm, Basic properties of SLE,
{\it Ann. Math. \bf 161} (2005), 879--920.



\bibitem[S]{S} S. Sheffield, Conformal weldings of random surfaces:
SLE and the quantum zipper, preprint.

\end{thebibliography}
\end{document}